\documentclass[11pt]{amsart}
\usepackage{amsmath,amsfonts,amsthm,amssymb}
\usepackage[pdfborder={0 0 0 [0 0 ]}]{hyperref}
\usepackage{enumerate}
\usepackage{mathrsfs}
\usepackage{enumerate}
\usepackage{graphicx}
\usepackage{color}
\usepackage[ps,all,arc,rotate,cmtip]{xy}
\usepackage{changepage}
\usepackage[capitalise]{cleveref}

\hypersetup{
    colorlinks=true,
    linkcolor=black,
    citecolor=black,
    filecolor=black,
    urlcolor=black,
}

\long\def\symbolfootnote[#1]#2{\begingroup
\def\thefootnote{\fnsymbol{footnote}}\footnote[#1]{#2}\endgroup}

\newtheorem*{theorem*}{Theorem}
\newtheorem{theorem}{Theorem}[section]

\newtheorem{lem}[theorem]{Lemma}
\newtheorem{thm}[theorem]{Theorem}
 \crefname{thm}{Theorem}{Theorems}

\newtheorem{prop}[theorem]{Proposition}

\newtheorem{cor}[theorem]{Corollary}
\crefname{cor}{Corollary}{Corollaries}

\theoremstyle{definition}

\newtheorem{rmk}[theorem]{Remark}

\newtheorem{dfn}[theorem]{Definition}

\newtheorem*{convention}{Convention}

\newtheoremstyle{maintheorem}{}{}{\itshape}{}{\bfseries}{}{.5em}{#1 \!\thmnote{#3}.}
\theoremstyle{maintheorem}
\newtheorem*{mainthm}{Theorem}

\newcommand{\R}{\mathbb{R}}
\newcommand{\Z}{\mathbb{Z}}

\newcommand{\Out}{\mathrm{Out}}
\newcommand{\Aut}{\mathrm{Aut}}

\renewcommand{\P}{\mathcal{P}}

\def\link-preserving{}

\renewcommand{\leq}{\leqslant}
\renewcommand{\geq}{\geqslant}

\def\im{\mathrm{im}}
\def\into{\hookrightarrow}
\def\iff{if and only if }
\def\wrt{with respect to }

\def\s-{\smallsetminus}

\def\Stab{\mathrm{Stab}}
\newcommand{\Isom}{\operatorname{Isom}}

\begin{document}

\title{Nielsen Realisation by Gluing: Limit Groups and Free Products}
\author{Sebastian Hensel}
\author{Dawid Kielak}\thanks{The second-named author was supported by the SFB 701}

\begin{abstract}
  \noindent We generalise the Karrass--Pietrowski--Solitar and the
  Nielsen realisation theorems from the setting of free groups to that
  of free products. As a result, we obtain a fixed point theorem for
  finite groups of outer automorphisms acting on the relative free
  splitting complex of Handel--Mosher and on the outer space of a free
  product of Guirardel--Levitt, as well as a relative version of the
  Nielsen realisation theorem, which in the case of free groups
  answers a question of Karen Vogtmann. We also prove Nielsen
  realisation for limit groups, and as a byproduct obtain a new proof that
  limit groups are CAT($0$).

  The proofs rely on a new version of Stallings' theorem on groups
  with at least two ends, in which some control over the behaviour of
  virtual free factors is gained.
\end{abstract}

\maketitle

\section{Introduction}
In its original form, the \emph{Nielsen realisation problem} asks
which finite subgroups of the mapping class group of a surface can be
realised as groups of homeomorphisms of the surface. A celebrated
result of Kerckhoff~\cite{Kerckhoff1980,Kerckhoff1983} answers this in the positive for all finite
subgroups, and even allows for realisations by isometries of a suitable
hyperbolic metric.

Subsequently, similar realisation results were found in other
contexts, perhaps most notably for realising finite groups in
$\Out(F_n)$ by isometries of a suitable graph (independently by
\cite{culler1984}, \cite{Khramtsov1985}, \cite{Zimmermann1981};
compare \cite{hop} for a different approach).

\medskip In this article, we begin to develop a \emph{relative}
approach to Nielsen realisation problems. The philosophy here is that
if a group $G$ allows for a natural decomposition into pieces, then
Nielsen realisation for $\Out(G)$ may be reduced to realisation in the
pieces, and a \emph{gluing problem}. In addition to just solving
Nielsen realisation for finite subgroups of $\Out(G)$, such an
approach yields more explicit realisations, which also exhibit the
structure of pieces for $G$.

We demonstrate this strategy for two classes of
groups: free products and limit groups. In another article, we use the
results presented here, together with the philosophy of relative Nielsen
realisation, to prove Nielsen realisation for certain right-angled Artin
groups (\cite{HenselKielak2016}).

\smallskip The early proofs of Nielsen realisation for free groups rely in a
fundamental way on a result of
Karrass--Pietrowski--Solitar~\cite{karrassetal1973}, which states that
every finitely generated virtually free group acts on a tree with
finite edge and vertex stabilisers. In the language of Bass--Serre
theory, it amounts to saying that such a virtually free group is a
fundamental group of a finite graph of groups with finite edge and vertex
groups.

This result of Karrass--Pietrowski--Solitar in turn relies on the
celebrated theorem of Stallings on groups with at least two
ends~\cite{Stallings1968, Stallings1971}. Stallings' theorem states that any
finitely generated group with at least two ends splits over a finite
group, which means that it acts on a tree with a single edge orbit and finite
edge stabilisers. Equivalently: it is a fundamental group of a
graph of groups with a single edge and a finite edge group.

\smallskip
In the first part of this article, we generalise these
results to the setting of a free product
\[ A = A_1 \ast \dots \ast A_n \ast B \]
in which we (usually) require the factors $A_i$ to be finitely generated, and $B$ to be a finitely generated free group.
Consider any finite group $H$ acting on $A$ by outer automorphisms in a way preserving the given free-product decomposition, by which we mean that  each element of $H$ sends each subgroup $A_i$ to some $A_j$ (up to conjugation); note that we do not require the action of $H$ to preserve $B$ in any way.
We then obtain a corresponding group extension
\[ 1\to A \to \overline A \to H \to 1\]
In this setting we prove (for formal
statements, see the appropriate sections)
\begin{description}
\item[Relative Stallings' theorem (\cref{prop: relative splitting})]
$\overline A$ splits over
a finite group, in such a way that each $A_i$ fixes a vertex in the
associated action on a tree.
\item[Relative Karrass--Pietrowski--Solitar
  theorem (\cref{KPS})] $\overline A$ acts on a
tree with finite edge stabilisers, and with each $A_i$ fixing a vertex
of the tree, and with, informally speaking, all other vertex groups
finite.
\item[Relative Nielsen realisation theorem (Theorem~\ref{rel NR})]
Suppose\\ that we are given complete non-positively
curved (i.e. locally CAT(0)) spaces $X_i$ realising the induced
actions of $H$ on the factors $A_i$. Then the action of
$H$ can be realised by a complete non-positively curved
space $X$; in fact $X$ can be chosen to contain the $X_i$
in an equivariant manner.
\end{description}
We emphasise that such a relative Nielsen realisation is new even if
all $A_i$ are free groups, in which case it answers a question of
Karen Vogtmann.

\smallskip
The classical Nielsen realisation for graphs  immediately implies
that a finite subgroup $H<\mathrm{Out}(F_n)$ fixes points in the Culler--Vogtmann Outer Space
(defined in~\cite{cullervogtmann1986}), as well as in the complex of
free splittings of $F_n$ (which is a simplicial closure of Outer
Space).
As an application of the work in this article, we similarly obtain fixed
point statements (\cref{fixed points,fixed point GL}) for the graph of relative free
splittings defined by Handel and Mosher~\cite{HandelMosher2014}, and the outer space of a free product defined by Guirardel and Levitt~\cite{GuirardelLevitt2007a}.

\smallskip
In the last section of the paper we prove
\begin{mainthm}[\ref{main: NR for limit groups}]
Let $A$ be a limit group, and let
\[
 A \to \overline A \to H
\]
be an extension of $A$ by a finite group $H$. Then there exists a complete compact locally CAT($\kappa$) space $X$ realising the extension $\overline A$, where $\kappa = -1$ when $A$ is hyperbolic, and $\kappa=0$ otherwise. When $\kappa=-1$, the space $X$ is of dimension at most $2$.
\end{mainthm}

This theorem is obtained by combining the classical Nielsen realisation theorems (for free, free-abelian and surface groups -- see \cref{NR for free groups,NR for abelian groups,NR}) with the existence of an invariant JSJ decomposition shown by Bumagin--Kharlampovich--Myasnikov~\cite{Bumaginetal2007}.

Note that, in general, having a graph of groups decomposition for a group $G$ with $\mathrm{CAT}(0)$ vertex groups and virtually cyclic edge groups does not allow one to build a $\mathrm{CAT}(0)$ space for $G$ to act on, and thus conclude that $G$ is itself $\mathrm{CAT}(0)$; the JSJ decompositions of limit groups are however special in this respect, and the extra structure allows for the conclusion. This has been observed by Sam Brown in \cite{Brown2016}, where he developed techniques for building up a $\mathrm{CAT}(0)$ space for $G$ to act on.

Observe that we obtain optimal curvature bounds for our space $X$ -- it has been proved by Alibegovi\'c--Bestvina~\cite{AlibegovicBestvina2006} that limit groups are CAT(0), and by Sam Brown~\cite{Brown2016} that a limit group is CAT($-1$) \iff it is hyperbolic.

Also, taking $H$ to be the trivial group gives a new (more direct) proof of the fact that limit groups are $\mathrm{CAT}(0)$.

\smallskip
Throughout the paper, we are going to make liberal use of the standard terminology of graphs of groups. The reader may find all the necessary information in Serre's book~\cite{serre2003}.
We are also going to make use of standard facts about
$\mathrm{CAT}(0)$ and non-positively curved (NPC) spaces, as well as more general $\mathrm{CAT}(\kappa)$ spaces; the standard
reference here is the book by
Bridson--Haefliger~\cite{bridsonhaefliger1999}.

\bigskip
\textbf{Acknowledgements.} The authors would like to thank Karen
Vogtmann for discussions and suggesting the statement of relative Nielsen realisation for free groups, Stefan Witzel for pointing out the work of Sam Brown, and the referee for extremely valuable comments.

\section{Relative Stallings' theorem}
\label{sec: rel ST}

In this section we will prove the relative version of Stallings' theorem.
Before we can begin with the proof, we need a number of definitions to formalise
the notion of a free splitting that is preserved by a finite group action.

\smallskip

\begin{convention}
When talking about free factor decompositions
\[A = A_1 \ast \dots \ast A_n \ast B\] of some group $A$, we will always assume
that at least two of the factors $\{ A_1, \dots, A_n, B \}$ are non-trivial.
\end{convention}

\begin{dfn}
\label{preserved}
Suppose that $\phi \colon H \to \Out(A)$ is a homomorphism with a finite
domain. Let $A = A_1 \ast \dots \ast A_n \ast B$ be a free factor
decomposition of $A$.  We say that this decomposition is
\emph{preserved by $H$} \iff for every $i$ and every $h \in H$, there is
some $j$ such that $h(A_i)$ is conjugate to $A_j$.

We say that a factor $A_i$ is \emph{minimal} \iff for any $h \in H$ the fact that $h(A_i)$ is conjugate to $A_j$ implies that $j \geqslant i$.
\end{dfn}

\begin{rmk}
 Note that when the decomposition is preserved, we obtain an induced action $H \to \mathrm{Sym}(n)$ on the indices $1, \dots, n$. We may thus speak of the stabilisers $\Stab_H(i)$ inside $H$.
 Furthermore, we obtain an induced action
 \[\Stab_H(i) \to \Out(A_i)\]

 The minimality of factors is merely a way of choosing a representative of each $H$ orbit in the action $H \to \mathrm{Sym}(n)$.
\end{rmk}

\begin{rmk}
  Given an action $\phi \colon H \to \Out(A)$, with $\phi$ injective
  and $A$ with trivial centre, we can define $\overline A \leqslant
  \Aut(A)$ to be the preimage of $H= \im \, \phi$ under the natural
  map $\Aut(A) \to \Out(A)$.
  We then note that $\overline A$ is an
  extension of $A$ by $H$:
  \[ 1 \to A \to \overline A \to H \to 1 \]
  and the left action of $H$ by outer
  automorphisms agrees with the left conjugation action inside the
  extension $\overline A$.

  Observe that then for each $i$ we also obtain an extension
  \[ 1 \to A_i \to \overline {A_i} \to \Stab_H(i) \to 1 \]
  where $\overline {A_i}$ is
  the normaliser of $A_i$ in $\overline A$.

  We emphasise that this construction works even when $A_i$
  itself is not centre-free. In this case it carries more
  information than the induced action $\Stab_H(i) \to \Out(A_i)$
  (e.g. consider the case of $A_i=\mathbb{Z}$ -- there are many different
  extensions corresponding to the same map to $\Out(\mathbb{Z})$).
\end{rmk}


%

We will now begin the proof of the relative version of Stallings' theorem.
It will use ideas from both Dunwoody's proof~\cite{Dunwoody1982} and
Kr\"on's proof~\cite{Kroen2010}\footnote{We warn the reader that later parts of
  Kr\"on's paper are not entirely correct; we only rely on the early, correct sections.} of Stallings'
theorem, which we now recall.

\begin{convention}
  If $E$ is a set of edges in a graph $\Theta$, we write $\Theta - E$ to mean
  the graph obtained from $\Theta$ by removing the interiors of edges in $E$.
\end{convention}

\begin{dfn}
  Let $\Theta$ be a graph. A finite subset $E$ of the edge set of
  $\Theta$ is called a set of \emph{cutting edges} \iff $\Theta - E$
  is disconnected and has at least two infinite components.

  A \emph{cut} $C$ is the union of all vertices contained in an
  infinite connected complementary component of some set of cutting
  edges. The \emph{boundary} of $C$ consists of all edges with exactly
  one endpoint in $C$.

  Given two cuts $C$ and $D$, we call them \emph{nested} \iff $C$ or
  its complement $C^*$ is contained in $D$ or its complement $D^*$. Note that $C^\ast$ and $D^\ast$ do not need to be cuts.
\end{dfn}

We first aim to show the following theorem which is implicit in
\cite{Kroen2010}.
\begin{thm}[\cite{Kroen2010}]
\label{thm: kroen}
Suppose that $\Theta$ is a connected graph on which a group $G$ acts.
Let $\mathcal P$ be a subset of the edge set of $\Theta$,
which is stable under the $G$-action.  If
there exists a set of cutting edges lying in $\mathcal P$, then there exists a
cut $C$ whose boundary lies in $\mathcal P$, such that $C^\ast$ is also a cut, and such that
furthermore for any $g \in G$ the cuts $C$ and $g.C$ are
nested.
\end{thm}

\begin{proof}[Sketch of proof]
  In order to prove this, we recall the following terminology, roughly
  following Dunwoody. We say that $C$ is a $\mathcal{P}$-cut \iff its
  boundary lies in $\mathcal{P}$. Say that a $\mathcal{P}$-cut is
  \emph{$\mathcal{P}$-narrow}, \iff its boundary contains the minimal
  number of elements among all $\mathcal{P}$-cuts.  Note that for each
  $\mathcal{P}$-narrow cut $C$, the complement $C^\ast$ is also a cut,
  as otherwise we could remove some edges from the boundary of $C$ and
  get another $\mathcal{P}$-cut.

Given any edge $e \in \mathcal{P}$, there are finitely many
$\mathcal{P}$-narrow cuts which contain $e$ in its boundary.
This is shown by Dunwoody \cite[2.5]{Dunwoody1982} for narrow cuts,
and the proof carries over to the $\mathcal{P}$-narrow
case. Alternatively, Kr\"on \cite[Lemma 2.1]{Kroen2010} shows this for sets of
cutting edges which cut the graph into exactly two connected
components, and $\mathcal{P}$-narrow cuts have this property.

Now, consider for each $\mathcal{P}$-narrow cut $C$ the number $m(C)$
of $\mathcal{P}$-narrow cuts which are not nested with $C$ (this is
finite by~\cite[2.6]{Dunwoody1982}). Call a $\mathcal{P}$-narrow cut
\emph{optimally nested} if $m(C)$ is smallest amongst all
$\mathcal{P}$-narrow cuts. The proof of Theorem~3.3 of
\cite{Kroen2010} now shows that optimally nested $\mathcal{P}$-cuts
are all nested with each other.
This shows Theorem~\ref{thm: kroen}.
\end{proof}

To use that theorem, recall

\begin{thm}[{\cite[Theorem 4.1]{Dunwoody1982}}]
\label{thm: dunwoody}
Let $G$ be a group acting on a graph $\Theta$.
Suppose that there exists a cut $C$, such that
\begin{enumerate}
 \item $C^\ast$ is also a cut; and
 \item there exists $g \in G$ such that $g.C$ is properly contained in $C$ or $C^\ast$; and
 \item $C$ and $h.C$ are
nested for any $h \in G$.
\end{enumerate}
Let $E$ be the boundary of $C$.
Then $G$ splits over the stabiliser of $E$
, and the stabiliser of any component of
$\Theta - G. E$ is contained in a conjugate of a vertex group.
\end{thm}

Now we are ready for our main splitting result.

\begin{figure}
\begin{center}
\includegraphics[scale=2]{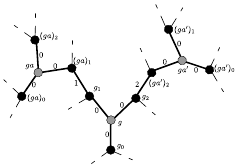}
\end{center}
\caption{A local picture of the graph $\Theta$.}
\label{Cayley graph}
\end{figure}

\begin{thm}[Relative Stallings' Theorem]
\label{prop: relative splitting}
Let $\phi \colon H \to \Out(A)$ be a monomorphism with a finite domain.
Let $A = A_1 \ast \dots \ast A_n \ast B$ be a free product decomposition with each $A_i$ and $B$ finitely generated, and suppose that it is preserved by $H$. Let $\overline A$ be the preimage of $H = \im \, \phi$ in $\Aut(A)$.
Then $\overline A$ acts on a tree with finite quotient so that
each $A_i$ fixes a vertex, and no non-trivial subgroup of $A$ fixes any edge.
\end{thm}
Note in particular that the quotient of the associated tree by $\overline A$ has a single edge.
\begin{proof}
  Before we begin the proof in earnest, we will give a brief outline
  of the strategy. First, we will define a variant of the Cayley graph
  for $\overline{A}$ in which the free product structure of $A$ will
  be visible (in fact, a subgraph will collapse to the Bass-Serre tree
  of the free product decomposition of $A$). This graph will contain
  the different copies if $A_i$ disjointly, separated by edges
  labelled with a certain label. We will then aim to show that there
  is a set of cutting edges just using edges with that label -- which,
  using Theorem~\ref{thm: dunwoody}, will yield the desired action on
  a tree.

\smallskip
Let $\mathcal A_i$ and $\mathcal B$ be finite generating sets of $A_i$ and $B$, respectively (for all $i\leqslant n$). We also choose a finite set $\mathcal H \subset \overline A$ which maps onto $H$ under the natural epimorphism $\overline A \to H$. Note that $\bigcup_i \mathcal A_i \cup \mathcal B \cup \mathcal H$ is a  generating set of $\overline A$.

We define $\Theta$ to be a variation of the (right) Cayley graph of
$\overline A$ with respect to the generating set $\bigcup_i
\mathcal A_i \cup \mathcal B \cup \mathcal H$. Intuitively, every vertex of the Cayley
graph will be ``blown up'' to a finite tree (see \cref{Cayley graph}). More formally,
the vertex set of $\Theta$ is
\[ V(\Theta) = \overline{A} \sqcup  (\overline{A}\times\{0,\ldots,n\}) \]
We adopt the notation that a vertex corresponding to an element in $\overline{A}$
will simply be denoted by $g$, whereas a vertex $(g,i)$ in the second factor will be
denoted by $g_i$.

We now define the edge set, together with a labelling of the edges by
integers $0, 1, \dots, n$, as follows:
\begin{itemize}
\item for each $g \in \overline A$ and each $i \in \{0, \dots, n\}$ we have an edge labelled by $0$ connecting $g$ to $g_i$;
\item for each $g \in \overline A$, each $i\geqslant 1$ and each $a \in \mathcal A_i$, we  have an edge labelled by $i$ from $g_i$ to $(ga)_i$;
\item for each $g \in \overline A$, and each $b \in \mathcal B \cup \mathcal H$, we  have an edge labelled by $0$ from $g_0$ to $(gb)_0$.
\end{itemize}
The group $\overline A$ acts on $\Theta$ on the left, preserving the
labels. The action is free and co-compact. The graph $\Theta$
retracts via a quasi-isometry onto a usual Cayley graph of $\overline
A$ by collapsing edges connecting $g$ to $g_i$. Also note that there are copies of the Cayley graphs of the $A_i$ with respect to the generating set $\mathcal{A}_i$ in $\Theta$, where each edge has the label $i$.

Let $\Omega$ denote a graph constructed in the same way for the group
$A$ with respect to the generating set $\bigcup \mathcal A_i \cup
\mathcal B$. There is a natural embedding of $\Omega$ into $\Theta$,
and hence we will consider $\Omega$ as a subgraph of $\Theta$. Note
that this embedding is also a quasi-isometry.

\medskip
We will now construct $n$ specific quasi-isometric retractions of $\Theta$ onto $\Omega$. These will be used later to modify paths in order to avoid edges with certain labels.

Let us fix $i \in \{ 1, \dots, n \}$. For each $h \in H$ we pick a representative $h_i \in \overline A$ thereof, such that $h_i A_i {h_i}^{-1} = A_j$ for a suitable (and unique) $j$; for $1 \in H$ we pick $1 \in \overline A$ as a representative. These elements $h_i$ are coset representatives of the normal subgroup $A$ of $\overline A$.

Such a choice defines a retraction $\rho_i \colon \Theta \to \Omega$ in the following way: each vertex $g$ is mapped to the unique vertex $g'$ where $g' \in A$ and $g'h_i =g$ for some $h_i$; the vertex $g_k$ is then mapped to $(g')_k$. An edge labelled by $0$ connecting $g$ to $g_k$ is sent to the edge connecting $g'$ to $(g')_k$. The remaining edges with label $0$ are sent in an $A$-equivariant fashion to paths connecting the image of their endpoints; the lengths of such paths are uniformly bounded, since (up to the $A$-action) there are only finitely many edges with label $0$.

Similarly,
the edges of label $k \not\in \{ 0, i\}$ are mapped in an $A$-equivariant manner to paths connecting the images of their endpoints; again, their length is uniformly bounded. 

Each edge labelled by $i$ is sent $A$-equivariantly to a path connecting the images of its endpoints, such that the path contains edges labelled only by some $j$  (where $j$ is determined by the coset of $A$ the endpoints lie in); such a path exist by the choice of the representatives $h_i$.

Note that each such retraction $\rho_i$ is a $(\kappa_i, \kappa_i)$-quasi-isometry for some $\kappa_i \geqslant 1$; we set $\kappa = \max_i \kappa_i$.

\medskip

Now we are ready to construct a set of cutting edges in $\Theta$.

Consider the ball $B_\Omega(1,1)$ of radius 1 around the vertex 1 in $\Omega$ (all of whose edges are labelled by $0$). Since $A$ is a nontrivial free product, the identity element disconnects the Cayley graph into at least two infinite components. Hence,
$B_\Omega(1,1)$ disconnects $\Omega$ also into at least two infinite components; let us take two vertices of $\Omega$, $x$ and $y$, lying in distinct infinite components of $\Omega - B_\Omega(1,1)$, and such that
\[d_\Omega(1,x) = d_\Omega(1,y) \geqslant \kappa^2 + 4\]

Now let $E$ denote the set of all edges lying in the ball $B_\Theta(1,\kappa^2 + 4)$ labelled by $0$. We claim that $E$ disconnects $\Theta$ into at least two infinite components.
Note that $\Theta - E$ has finitely many components, since $E$ is finite. By possibly choosing $x,y$ even further from each other, it therefore suffices to show that
$E$ disconnects $x$ from $y$ (viewed as vertices of $\Theta$).

Suppose for a contradiction that there exists a path $\gamma$ in $\Theta - E$ connecting $x$ to $y$. Using any of the quasi-isometries $\rho_i$ we immediately see that $\gamma$ has to go through $B_\Theta(1,\kappa^2 + 4)$, since $\rho_i(\gamma)$ must intersect $B_\Omega(1,1)$.
Note that if $\gamma' \subset \gamma$ is a subpath lying completely in $B_\Theta(1,\kappa^2+4)$, then $\gamma'$ only traverses edges with the same label (as $\gamma$ does not intersect $E$). Thus, we can write $\gamma$ as a concatenation
\[ \gamma = \gamma_1 * \dots * \gamma_m \]
where each $\gamma_i$ intersects $B_\Theta(1,\kappa^2 + 4)$ only at edges of one label, and its endpoints lie outside of $B_\Theta(1,\kappa^2 + 4)$. We modify each $\gamma_i$ by pre- and post-concatenating it with a path of length at most $4$ (note that all the elements of $\mathcal{H}$ correspond to edges), so that it now starts and ends at $\Omega$. Still, the new path (which we will continue to call $\gamma_i$) intersects $B_\Theta(1, \kappa^2 + 1)$ only at edges labelled by a single label. 

Now we construct a new path $\gamma'$ as follows. Suppose that $k_i$ is such that
each edge in $\gamma_i \cap B_\Theta(1,\kappa^2+1)$ has label $k_i$. We put
\[ \gamma_i' = \rho_{k_i}(\gamma_i)\]
Note that as $\rho_{k_i}$ is a retraction onto $\Omega$, and the endpoints of $\gamma_i$ are in $\Omega$, the path $\gamma_i'$ has the same endpoints as $\gamma_i$. Put
\[ \gamma' = \gamma_1' * \dots * \gamma_m' \]
This is now a path joining $x$ to $y$ in $\Omega$, and thus contains an edge \[e \in B_\Omega(1,1)\]

There exists an edge $f$ in some $\gamma_i$, such that $e$ lies in the image of $f$ under the map $\rho_{k_i}$ that we applied to $\gamma_i$. Since each $\rho_k$ is an $(\kappa,\kappa)$-quasi-isometry, the edge $f$ lies within $B_\Theta(1,\kappa^2 + 1)$. But then $\rho_{k_i}(f)$ is a path the edges of whom are never labelled by $0$, and so in particular $e \not\in E$, a contradiction.

\smallskip
We now apply Theorem~\ref{thm: kroen}, taking $\mathcal P$ to be the set of edges labelled by $0$. Let $C$ denote the cut we obtain, and let $F$ denote its boundary.

To apply Theorem~\ref{thm: dunwoody} we need to only show that for some $g \in \overline A$ we have $g.C$ properly contained in $C$ or $C^\ast$.
Since $C^\ast$ is infinite, it contains an element $g \in \overline A$ such that $g.F \neq F$. Taking such a $g$, we see that either $g.C$ is properly contained in $C^\ast$ (in which case we are done), or $C$ is properly contained in $g.C$. In the latter case we have $g^{-1} .C \subset C$. We have thus verified all the hypotheses of Theorem~\ref{thm: dunwoody}.

Since the boundary $F$ of the final cut $C$ is labelled by $0$, upon removal of the open edges in $\overline A . F$, the connected component containing $1_i$ contains the entire subgroup $A_i$, since vertices corresponding to elements of this subgroup are connected to $1_i$ by paths labelled by $i$. Thus $A_i$ is a subgroup of a conjugate of a vertex group, and so it fixes a vertex in the associated action on a tree.

\smallskip It remains to show the
triviality of edge stabilisers in $A$. In fact we will show that no non-trivial subgroup
$G < A$ fixes a narrow cut in $\Theta$ with boundary consisting only of edges labelled by $0$. To
this end, let $C$ be such a cut, and $F$ the set of edges forming the
boundary of $C$.

We begin by considering the subgraph $\Omega$.
Let $\Gamma$ be an infinite component of $\Omega - F$, and $h\in H$ be
arbitrary. There are infinitely many vertices $v$ in $\Gamma$ such that
no edge emanating from $v$ lies in $F$ (as the latter is finite). Take
one such vertex, and consider an edge $e$ in its star which
corresponds to right multiplication with $h$. Since $h$ normalises $A$, it in fact connects
$\Omega$ to $h.\Omega$. On the other hand, there can be only a single component of
$h.\Omega - F$ which is connected to $\Gamma$ as the cut $C$ is narrow: otherwise
the components of $h.\Omega - F$ would lie in the same component of $\Theta-F$, and
$F$ would fail the definition of a boundary of a cut.

In summary, we have shown, that for each $h$, each infinite
component $\Gamma$ of $\Omega-F$ is connected (via an edge corresponding to right multiplication by
$h$) to a unique infinite component of $h.\Omega-F$. In other words,
infinite components of $\Omega-F$ and $h.\Omega-F$ are in bijection to
each other, where the bijection identifies components which are connected in $\Theta-F$.

Now, we can think of $\Omega$ as the Bass-Serre tree for the splitting
of $A$, whose vertices have been ``blown up'' to Cayley graphs of the
subgroups $A_i$. In particular, each edge labelled by $0$ disconnects
$\Omega$. This implies that $\Omega-F$, and hence each $h.\Omega - F$,
has exactly two components, both of which are infinite. Namely, if
$\Omega-F$ would have more than two infinite components, or just a single
one, the same would be true for $\Theta-F$, violating narrowness of
the cut $F$. It also implies that $F \cap h.\Omega$ consists of exactly one edge for each $h$.
Since $A$ acts freely on $\Omega$,
this
implies the final claim of the theorem.
\end{proof}

\section{Blow-ups}

We make the convention that graphs of groups are always connected unless explicitly stated otherwise.

\begin{prop}[Blow-up with finite edge groups]
\label{blow up}
Let $G$ be a graph of groups with finite edge groups. For each vertex $v$ suppose that the associated vertex group $G_v$ acts on a connected space $X_v$ in such a way that
each finite subgroup of $G_v$ fixes a point of $X_v$.
Then there exists a connected space $Y$ on which $\pi_1(G)$ acts, satisfying the following:
\begin{enumerate}
\item there is a $\pi_1(G)$-equivariant map $\pi \colon Y \to \widetilde G$;
\item if $w$ is a vertex of $\widetilde G$ fixed by $G_v$, then $\pi^{-1}(w)$ is $G_v$-equivariantly isometric to $X_v$;
\item every finite subgroup of $G$ fixes a point of $Y$.
\end{enumerate}
Moreover, when the spaces $X_v$ are compact, complete and CAT(0) then $Y$ is a compact, complete CAT(0) space.
\end{prop}
\begin{proof}
Recall that the vertices of $\widetilde G$ are left cosets of the vertex groups $G_v$ of $G$; for each vertex $w$ we pick an element $z_w \in G$ to be a  coset representative of such a coset.

We will build the space $Y$ in two steps. First, we construct the preimage under $\pi$ of the vertices of $\widetilde G$, and call it $V$. We define $V$ to be the disjoint union of spaces $X_w$, where $w$ runs over the vertices of $\widetilde G$, and $X_w$ is an isometric copy of $X_v$, where $v$ is the image of $w$ under the quotient map $\widetilde G \to G$.
We construct $\pi \colon V \to \widetilde G$ by declaring $\pi(X_w) = \{w\}$.

We now construct an action of $A = \pi_1(G)$ on $V$. Let us take $X_w \subset V$, and let $a \in A$. Let $u = a.w$, and note that its image in $G$ is still $v$. The action of $a$ on $V$ will take $X_w$ to $X_u$; using the identifications $X_w \simeq X_v \simeq X_u$ we only need to say how $a$ is supposed to act on $X_v$, and here it acts as $z_w^{-1} a$.

\smallskip
We now construct the space $Y$ by adding edges to $V$.

Let $e$ be an edge of $\widetilde G$ with terminal endpoint $w$ and initial endpoint $u$.
Let $X_e$ denote a copy of the unit interval. Now $G_e$ is a finite subgroup of $G_w$, and so fixes a point in $X_w$ seen as a subset of $V$. We glue the endpoint 1 of $X_e$ to this point. Analogously, we glue the endpoint 0 to a point in $X_u$. Now, using the action of $A$, we equivariantly glue all the endpoints of the edges in the $A$-orbit of $e$. We proceed this way for all (geometric) edges. Note that this construction allows us to extend the definition of $\pi$.

\smallskip

When all the vertex spaces are complete CAT(0), it is clear that so is $Y$.
\end{proof}

\begin{rmk}
 Suppose that the spaces $X_v$ in the above proposition are trees. Then the resulting space $Y$ is a tree, and the quotient graph of groups is obtained from $G$ by replacing $v$ by the quotient graph of groups $X /\!\!/ G_v$.
\end{rmk}

We will refer to the above construction as \emph{blowing up} $G$ by the spaces $X_v$. We warn the reader that our notion of a blow-up is not standard terminology
(and has nothing to do with blow-ups in other fields).

\medskip

When dealing with limit groups, we will need a more powerful version of a blow-up. We will use a method by Sam Brown, essentially following \cite[Theorem 3.1]{Brown2016}; to this end let us start with a number of definitions and standard facts.

\begin{dfn}
An \emph{$n$-simplex of type $M_\kappa$} is the convex hull of $n+1$ points in general position lying in the $n$-dimensional model space $M_\kappa$ of curvature $\kappa$, as defined in \cite{bridsonhaefliger1999}.

An $M_\kappa$-simplicial complex $K$ is a simplicial complex in which each simplex is endowed with the metric of a simplex of type $M_\kappa$, and the face inclusions are isometries.
\end{dfn}
Note that we will be interested in the case of $n=2$ and negative $\kappa$, where the model space $M_{\kappa}$ is just a suitably rescaled hyperbolic plane.

\begin{dfn}
Let $K$ be a $M_\kappa$-simplicial complex of dimension at most 2. The \emph{link} of a vertex $v$ is a metric graph whose vertices are edges of $K$ incident at $v$, and edges are $2$-simplices of $K$ containing $v$. Inclusion of edges into simplices in $X$ induces the inclusion of vertices into edges in the link. The length of an edge in the link is equal to the angle the edges corresponding to its endpoints make in the simplex.
\end{dfn}

Let us state a version of Gromov's link condition adapted to our setting.

\begin{thm}[Gromov's link condition~{\cite[Theorem II.5.2]{bridsonhaefliger1999} }]
\label{gromov}
Let $K$ be a $M_\kappa$-simplicial complex of dimension at most 2, endowed with a cocompact simplicial isometric action. Then $K$ is a locally $\mathrm{CAT}(\kappa)$ space \iff the link of each vertex in $K$ is $\mathrm{CAT}(1)$.
\end{thm}
Of course, for a graph being $\mathrm{CAT}(1)$ is equivalent to having no non-trivial simple loop of length less than $2 \pi$.

\begin{lem}[{\cite[Lemma 2.29]{Brown2016}}]
\label{sam's lemma}
For any $0<\theta<\pi$ and any  $A, C$  with $C>A>0$, there exists $k<0$ and a locally $\mathrm{CAT}(k)$
$M_k$-simplicial annulus with one locally geodesic boundary component of length  $A$, and one
boundary  component  of  length $C$  which  is  locally  geodesic  everywhere  except  for  one  point
where it subtends an angle greater than $\theta$.
\end{lem}

\begin{lem}
\label{v cyclic}
Let $Z$ be an infinite virtually cyclic group. Any two cocompact isometric actions on $\R$ have the same kernel, and the quotient of $Z$ by the kernel is isomorphic to either $\Z$ or the infinite dihedral group $D_\infty$.
\end{lem}
\begin{proof}
Clearly both actions on $\R$ can be made into actions on 2-regular trees with a single edge orbit and no edge inversions; each such action gives us a decomposition of $\Z$ into a graph of finite groups, where the kernel of the action is the unique edge group, and the quotient is as claimed. Let $G_1$ and $G_2$ denote the graphs of groups, and $K_1$ and $K_2$ denote the respective edge groups.

Suppose that one of the graphs, say $G_1$, has only one vertex. Then $K_1$ is also equal to the vertex group, and we have $K_2 \leqslant K_1$, since any finite group acting on a tree has a fixed point. If $G_2$ also has a single vertex than $K_1 \leqslant K_2$ by the same argument and we are done. Otherwise $Z/K_1\simeq \Z$ is a quotient of $Z/K_2 \simeq D_\infty$, which is impossible.

Now suppose that both $G_1$ and $G_2$ have two vertices each. Let $G_v$ be a vertex group of $G_1$. Arguing as before we see that it fixes a point in the action of $Z$ on $\widetilde G_2$, and so some index 2 subgroup of $G_v$ fixes an edge. Thus $K_1 \cap K_2$ is a subgroup of $K_1$ of index at most two. If the index is two, then the image of $K_1$ in $Z/K_2 \simeq D_\infty$ is a normal subgroup of cardinality 2. But $D_\infty$ does not have such subgroups, and so $K_1 \leqslant K_2$. By symmetry $K_2 \leqslant K_1$ and we are done.
\end{proof}

Let us record the following standard fact.
\begin{lem}
\label{axis}
Let $Z$ be an infinite virtually cyclic group acting properly by semi-simple isometries on a complete $\mathrm{CAT}(0)$ space $X$. Then $Z$ fixes an image of a geodesic in $X$ (called an \emph{axis}).
\end{lem}

For the purpose of the next proposition, let us introduce some notation.

\begin{dfn}
A CAT($-1$) $M_{-1}$-simplicial complex of dimension at most $2$ with finitely many isometry classes of simplices  will be called \emph{useful}.
\end{dfn}

\begin{prop}[Blow-up with virtually cyclic edge groups]
\label{blow up 2}
Suppose that $\kappa \in \{0,-1\}$.
Let $G$ be a finite graph of groups with virtually cyclic edge groups. For each vertex $v$ suppose that the associated vertex group $G_v$ acts properly cocompactly on a connected complete $\mathrm{CAT}(\kappa)$ simplicial complex $X_v$ by semi-simple isometries. Suppose further that
\begin{enumerate}
\item[(A1)] there exists an orientation of geometric edges of $G$ such that the initial vertex of every edge $e$ is \emph{useful}: that is, it is a vertex $u$ with $X_{u}$ useful; and
 \item[(A2)] when $X_u$ is useful and $e_1, \dots, e_n$ are all the edges of $G$ incident at $u$ carrying an infinite edge group, then the axes preserved by $g^{-1} X_{e_i} g$ with $i \in \{1, \dots , n\}$ and $g_i \not \in X_{e_i}$
can be taken to be simplicial and pairwise transverse.
 \end{enumerate}
Then there exists a connected complete $\mathrm{CAT}(\kappa)$ space $Y$ on which $\pi_1(G)$ acts cocompactly, satisfying the following:
\begin{enumerate}
\item there is a $\pi_1(G)$-equivariant map $\pi \colon Y \to \widetilde G$;
\item if $w$ is a vertex of $\widetilde G$ fixed by $G_v$, then $\pi^{-1}(w)$ is $G_v$-equivariantly isometric to $X_v$.
\end{enumerate}
\end{prop}
\begin{proof}
We will proceed exactly as in the proof of \cref{blow up}, with two exceptions: firstly, we will rescale the spaces $X_v$ before we start the construction; secondly, we will need to deal with infinite virtually cyclic edge groups.
Let us first explain how to deal with the infinite edge groups, and then it will become apparent how we need to rescale the useful spaces.

Let $e$ be an oriented edge of $\widetilde G$ with infinite stabiliser $G_e$ (note that this is a slight abuse of notation, as we usually reserve $G_e$ to be an edge group in $G$ rather than a stabiliser in $\widetilde G$). The group is virtually cyclic, and so, by \cref{axis}, fixes an axis in each of the vertex spaces corresponding to the endpoints of $e$ (it could of course be two axes in a single space, if $e$ is a loop). The actions on these axes are equivariant by \cref{v cyclic}, and the only difference is the length of the quotient of the axis by $G_e$; we will denote the two lengths by $\lambda^+_e$ and $\lambda^-_e$, where $\lambda^+_e$ is the amount by which $G_e$ translates the axis corresponding to the terminus of $e$, and $\lambda^-_e$ to the origin.

We claim
that we can rescale the spaces $X_v$ and orient the geometric edges so that for any edge $e$ with infinite stabiliser we have the initial vertex of $e$ useful and $\lambda_e^+ \leqslant \lambda_e^-$.
Let us assume that we have already performed a suitable rescaling -- we will come back to it at the end of the proof.

Let $u$ denote the initial (useful) endpoint of $e$; let $w$ denote the other endpoint of $e$.
We replace each $2$-dimensional simplex in $X_u$ by the comparison simplex of type $M_{-\frac 1 2}$ -- note that, in particular, this does not affect the metric on the $1$-skeleton of $X_u$, and hence does not affect the constant $\lambda_e^-$. Let $\widehat X_u$ denote the resulting space.

In $\widehat X_u$ we have generated, in Brown's terminology, an
\emph{excess angle} $\delta$ (depending on $u$), that is in the link
of any vertex $x$ in $\widehat X_u$ the distance between any two
points which were of distance at least $\pi$ in the link of $x$ in
$X_u$ is at least $\pi +2 \delta$ in the link in $\widehat X_u$.  By
possibly decreasing $\delta$, we may assume that $\delta < \frac \pi
3$, and that the distance between any two distinct vertices in a link
of a vertex in $\widehat X_u$ is at least $\delta$ (this is possible
since there are only finitely many different isometry types of
simplices in $X_u$, and so in $\widehat X_u$).  We still have $G_u$
acting on $\widehat X_u$ simplicially and isometrically.

Suppose that $\lambda^+_e = \lambda^-_e$. Then we take $X_e$ to be a flat strip $[0,1]\times \R$ on which $G_e$ acts by translating the $\R$ factor so that the quotient is isometric to $[0,1]\times \R/\lambda^+_e \Z$.

If $\lambda^+_e \neq \lambda^-_e$ then
we take $X_e$ to be the universal cover of an annulus from \cref{sam's lemma} with boundary curves of length $\lambda_e^+$ and $\lambda_e^-$, and $\theta = \pi-\delta$. The space $X_e$ is a $\mathrm{CAT}(k_e)$
$M_{k_e}$-simplicial complex  for some $k_e<0$.

We glue the preimage (in $X_e$) of each of the boundary curves to the corresponding axis of $G_e$, so that the gluing is an $G_e$-equivariant isometry. The gluing along the preimage of the shorter curve (or both curves if they are of equal length) proceeds along convex subspaces, and so if the vertex space was $\mathrm{CAT}(\mu)$ with $\mu \leqslant 0$, then the glued-up space is still locally $\mathrm{CAT}(\mu)$ along the axis of $G_e$.

The situation is different at the useful end: here we glue in along a non-convex curve. We claim that the resulting space is still locally CAT($k_e$) along this geodesic. This follows from Gromov's link condition (\cref{gromov}), and the observation that in the link of any vertex of $\widehat X_u$ we introduced a single path (a \emph{shortcut}) of length at least $\pi - \delta$ between vertices whose distance before the introduction of the shortcut was at least $\pi + 2 \delta$. A simple closed curve which traverses both endpoints of the shortcut therefore had length at least $2\pi + 4\delta$ before introducing the shortcut, and thus still has length $\geq 2\pi+\delta$ afterwards.
Thus there is still no non-trivial simple loop shorter than $2 \pi$.

We now use the action of $A = \pi_1(G)$ to equivariantly glue in copies of $X_e$ for all edges in the orbit of $e$. We proceed in the same way for all the other (geometric) edges.

Now we need to look at the curvature. The useful spaces have all been altered to be $M_{-\frac 1 2}$-simplicial complexes, and so they are now $\mathrm{CAT}(-\frac 1 2)$.
If we had any $\mathrm{CAT}(0)$ vertex spaces, then they remain $\mathrm{CAT}(0)$. The universal covers $X_e$ of annuli are $\mathrm{CAT}(k_e)$ with $k_e<0$; the infinite strips are $\mathrm{CAT}(0)$. The gluing into the non-useful spaces did not disturb the curvature. A single gluing into a useful space did not disturb the curvature either, but the situation is more complicated when we glue more than one space $X_e$ into a single $\widehat X_u$, since we could have introduced multiple shortcuts of length at least $\pi - \delta$ into a link of a single vertex.
If a curve traverses one (or no) shortcut, then the argument given above shows that it has length at least $2\pi$. If it traverses more more than $2$, then (as $\delta<\pi/3$), it also has length $\geq 2\pi$. In the final case where it goes through exactly two, note
that the endpoints of the shortcuts are all distinct by the transversality assumption (A2). Hence, by the choice of $\delta$, any path connecting these endpoints has length $>\delta$, and so the total path has length $> 2(\pi-\delta) +2\delta$ as well.

We conclude that our space $Y$ is complete and CAT($k$), where $k$ is the maximum of the  values $k_e, \kappa$ and $-\frac 1 2$. When $\kappa = 0$ we have $k=0$ and we are done. Otherwise, observing that we had only finitely many edges in $G$, we have $k<0$, and so we can rescale $Y$ to obtain a $\mathrm{CAT}(-1)$ space, as claimed.

\smallskip
We still need to explain how to rescale the vertex spaces. We order the vertices of the graph of groups $G$ in some way, obtaining  a list $v_1, \dots, v_m$. The space $X_{v_1}$ we do not rescale. Up to reorienting the geometric edges running from $v_1$ to itself we see that the constants $\lambda_e^+$ and $\lambda_e^-$ for such edges satisfy $\lambda_e^+ \leqslant \lambda_e^+$.

We look at the full subgraph $\Gamma$ of $G$ spanned by the vertices $v_1, \dots, v_i$. Inductively, we assume that the spaces corresponding to vertices in $\Gamma$ have already been rescaled as required. Now we attach $v_{i+1}$ to $\Gamma$, together with all edges connecting $v_{i+1}$ to itself or $\Gamma$. If $X_{v_{i+1}}$ is not useful, then we have no edges of the latter type, and all edges connecting $v_{i+1}$ to $\Gamma$ are oriented towards $v_{i+1}$. Clearly we can rescale $X_v$ to be sufficiently small so that the desired inequalities are satisfied (note that there are only finitely many edges to consider).

If $X_{v_{i+1}}$ is useful then we can reorient all edges connecting $v_{i+1}$ to $\Gamma$ so that they run away from $v_{i+1}$. Now we can make $X_{v_{i+1}}$ sufficiently big to satisfy the desired inequalities. We also reorient the edges connecting $v_{i+1}$ to itself in a suitable manner.
\end{proof}

\section{Relative Karrass--Pietrowski--Solitar theorem}

The following theorem is a generalisation of a theorem of Karrass--Pietrowski--Solitar~\cite{karrassetal1973}, which lies behind the Nielsen realisation theorem for free groups.

\begin{thm}[Relative Karrass--Pietrowski--Solitar theorem]
\label{KPS}
Let
\[\phi \colon H \to \Out(A)\]
 be a monomorphism with a finite domain, and let
\[A = A_1 \ast \dots \ast A_n \ast B\]
be a decomposition preserved by $H$, with each $A_i$ finitely generated, non-trivial, and $B$ a (possibly trivial) finitely generated free group.
Let $A_1, \dots, A_m$ be the minimal factors.
Then the associated extension $\overline A$ of $A$ by $H$ is isomorphic to the fundamental group of a finite graph of groups with finite edge groups, with $m$ distinguished vertices $v_1, \dots, v_m$, such that the vertex group associated to $v_i$ is a conjugate of the extension $\overline{A_i}$ of $A_i$ by $\Stab_H(i)$, and vertex groups associated to other vertices are finite.
\end{thm}
\begin{proof}
The proof goes along precisely the same lines as the original proof of Karrass--Pietrowski--Solitar~\cite{karrassetal1973}, with the exception that we use Relative Stallings' Theorem (\cref{prop: relative splitting}) instead of the classical one.

We will prove the result by an induction on a
\emph{complexity} $(n, f)$ where $n$ is the number of factors $A_i$,
and $f$ is the rank of the free group $B$ in the decomposition.
We order the complexity lexicographically.
 The cases of complexity $(0, f)$ follow from the usual Nielsen realisation theorem for free groups (see Theorem~\ref{NR for free groups}).

\smallskip
Thus, for the inductive step, we assume a complexity $(m, f)$ with $m>0$.
We begin by applying \cref{prop: relative splitting} to the
finite extension $\overline A$. We obtain a graph of groups $P$ with
one edge and a finite edge group, such that each $A_i$ lies up to
conjugation in a vertex group, and no non-trivial subgroup of any factor $A_i$ fixes an edge.

Let $v$ be any vertex of $\widetilde P$. The group $P_v$ is a finite extension of $A \cap P_v$ by a subgroup $H_v$ of $H$.
Let us look at the structure of $P_{v} \cap A$ more closely.

Consider the graph of groups associated to the product $A_1 \ast
\dots A_n \ast B$ and apply Kurosh's theorem~\cite[Theorem I.14]{serre2003} to the subgroup $P_{v} \cap
A$. We obtain that $P_{v} \cap A$ is a free product of groups of the form $P_{v}
\cap x A_i x^{-1}$ for some $x \in A$, and a free group $B'$.

Let us suppose that the intersection $P_{v} \cap x A_i x^{-1} $ is nontrivial for some $i$ and $x \in A$. This implies that a non-trivial subgroup $G$ of $A_i$ fixes the vertex $x^{-1} .v$. We also know that $A_i$ fixes some vertex $v_i$ in $\widetilde P$ by construction, and thus so does $G$.
If $x^{-1}.v \neq v_i$, this would imply that $G$ fixes an edge, which is impossible.
Hence $v_i = x^{-1} .v$ and in particular we have that $x A_i x^{-1} \leqslant P_v$.

Now suppose that $P_{v} \cap y A_i y^{-1}$ is non-trivial for some
other element $y \in A$. Then $x^{-1}.v = v_i = y^{-1} .v$, and so $x
y^{-1} \in A \cap P_v$. This implies that the two free factors $P_{v}
\cap x A_i x^{-1}$ and $P_{v} \cap y A_i y^{-1}$ of $P_v \cap A$ are
conjugate inside the group, and so they must coincide.


\smallskip We consider the action of $A$ on the tree $\tilde{P}$, and conclude that $A$ is
equal to the fundamental group of the graph of groups $\tilde{P}/\!\!/A$.
The discussion above shows that:
\begin{enumerate}[i)]
\item The stabilizer of a vertex $v \in \tilde P$ has the structure
\[ P_v \cap A = x_{i(v,1)} A_{i(v,1)} x_{i(v,1)}^{-1} \ast \dots \ast
x_{i(v,k)} A_{i(v,k)} x_{i(v,k)}^{-1} \ast B' \]
where the indices $i(v,k)$ are all distinct, and $B'$ is some free group.
\item If a conjugate of $A_i$ intersects some stabilizer of $v$ non-trivially, then it stabilizes $v$.
\item For each $i$ there is exactly one vertex $v$ so
  that a conjugate of $A_i$ appears as $A_{i(v,l)}$ in the description above.
\item The edge groups in $\tilde{P}/\!\!/A$ are trivial.
\end{enumerate}
Since the splitting which $P$ defines is non-trivial, the index of
$P_v \cap A$ in $\overline A$ is infinite, and thus $A$ is not a
subgroup of $P_v$ for any $v$.

Next, we aim to show that the complexity of each $P_v \cap A$ is
strictly smaller than that of $A$. To begin, note that the only way
that this could fail is if there is some vertex $w$ so that
\[ P_w \cap A = x_1 A_1 x_1^{-1} \ast \dots \ast x_m A_m x_m^{-1} \ast B' \]
for $B'$ a free group. Since all edge groups in $\tilde{P}/\!\!/A$ are trivial, $A$ is obtained
from $P_w \cap A$ by a free product with a free group. Such an operation cannot decrease
the rank of $B'$, and in fact increases it unless the free product is trivial. But in the latter
case we would have $P_w \cap A = A$, which is impossible.

We have thus shown that each $P_v$ is an extension
\[ P_v \cap A \to P_v \to H_v \]
where $H_v$ is a subgroup of $H$, the group $P_v \cap A$ decomposes in a way which is preserved by $H_v$, and its complexity is smaller than that of $A$. Therefore the group $P_v$ satisfies the assumption of the inductive hypothesis.

We now use \cref{blow up} (together with the remark following it) to construct a new graph of groups $Q$, by blowing $P$ up at $u$ by the result of the theorem applied to $P_u$, with $u$ varying over some chosen lifts of the vertices of $P$.

By construction, $Q$ is a finite graph of groups with finite edge
groups, and the fundamental group of $Q$ is indeed $\overline A$.
Also, $Q$ inherits distinguished vertices from the graphs of groups we blew up with.
Thus, $Q$ is as required in the assertion of our theorem, with two possible exceptions.

Firstly, it might have too many distinguished vertices.
This would happen if for some $i$ and $j$ we have $A_i$ and $A_j$ both being subgroups of, say, $P_v$, which are conjugate in $\overline A$ but not in $P_v$.
Let $h \in \overline A$ be an element such that $h A_i h^{-1} = A_j$.
Since both $A_i$ and $A_j$ fix only one vertex, and this vertex is $v$, we must have $h \in \P_v$, and so $A_i$ and $A_j$ are conjugate inside $P_v$.

Secondly, it could be that the finite extensions of $A_i$ we obtain as vertex groups are not extensions by $\Stab_H(i)$. This would happen if $\Stab_H(i)$ is not a subgroup of $H_v$. Let us take $h \in \overline A$ in the preimage of $\Stab_H(i)$, such that $h A_i h^{-1} = A_i$. Then in the action on $\widetilde P$ the element $h$ takes a vertex fixed by $A_i$ to another such; if these were different, then $A_i$ would fix an edge, which is impossible. Thus $h$ fixes the same vertex as $A_i$. This finishes the proof.
\end{proof}

\section{Fixed points in the graph of relative free splittings}

Consider a free product decomposition
\[ A = A_1 \ast \dots \ast A_n \ast B\]
with $B$ a finitely generated free group.
Handel and Mosher~\cite{HandelMosher2014} (see also the work of Horbez~\cite{Horbez2014}) defined a \emph{graph of relative free splittings} $\mathcal{FS}(A, \{A_1, \dots, A_n\})$ associated to such a decomposition. Its vertices are finite non-trivial graphs of groups with trivial edge groups, and such that each $A_i$ is contained in a conjugate of a vertex group; two such graphs of groups define the same vertex when the associated universal covers are $A$-equivariantly isometric. Two vertices are connected by an edge \iff the graphs of groups admit a common refinement.

In their article, Handel and Mosher prove that $\mathcal{FS}(A, \{A_1, \dots, A_n\})$ is connected and Gromov hyperbolic~\cite[Theorem 1.1]{HandelMosher2014}.

Observe that the subgroup $\Out(A, \{A_1, \dots, A_n \})$ of $\Out(A)$ consisting of those outer automorphisms of $A$ which preserve the decomposition
\[ A = A_1 \ast \dots \ast A_n \ast B\]
acts on this graph.
We offer the following fixed point theorem for this action on $\mathcal{FS}(A, \{A_1, \dots, A_n\})$.

\begin{cor}
\label{fixed points}
Let $H \leqslant \Out(A, \{A_1, \dots, A_n \})$ be a finite subgroup, and suppose that the factors $A_i$ are finitely generated.
Then $H$ fixes a point in the free-splitting graph $\mathcal{FS}(A, \{A_1, \dots, A_n \})$.
\end{cor}
\begin{proof}
\cref{KPS} gives us an action of the extension $\overline A$ on a tree $T$; in particular $A$ acts on this tree, and this action satisfies the definition of a vertex in $\mathcal{FS}(A, \{A_1, \dots, A_n \})$.
Since the whole of $\overline A$ acts on $T$, every outer automorphism in $H$ fixes this vertex.
\end{proof}

\section{Fixed points in the outer space of a free product}

Take any finitely generated group $A$, and consider its \emph{Grushko decomposition}, that is a free splitting
\[
 A = A_1 \ast \dots \ast A_n \ast B
\]
where $B$ is a finitely generated free group, and each group $A_i$ is finitely generated and freely indecomposible, that is it cannot act on a tree without a global fixed point (note that $\Z$ is not freely indecomposible in this sense).

Grushko's Theorem~\cite{Gruschko1940} tells us that such a decomposition is essentially unique; more precisely, if
\[
 A = A_1' \ast \dots \ast A_m' \ast B'
\]
is another such decomposition, then $B \cong B'$, $m=n$, and there is a permutation $\beta$ of the set $\{1, \dots, n\}$ such that $A_i$ is conjugate to $A_{\beta(i)}'$. In particular, this implies that the decomposition
\[
 A = A_1 \ast \dots \ast A_n \ast B
\]
is preserved in our sense by every outer automorphism of $A$.

In~\cite{GuirardelLevitt2007a} Guirardel and Levitt introduced $P \mathcal O$, the (projectivised) outer space of a free product. It has a structure of a `simplicial complex with missing faces' -- it is homeomorphic to a union of open simplices in a metric realisation of a simplicial complex. In particular, each open simplex contains a barycentre; the barycentres are
equivalence classes of pairs $(G, \iota)$, where:
\begin{enumerate}
 \item $G$ is a finite graph of groups with trivial edge groups;
 \item edges of $G$ are given length $1$;
 \item for every $i \in \{1, \dots, n \}$, there is a unique vertex $v_i$ in $G$ such that the vertex group $G_{v_i}$ is conjugate to $A_i$;
 \item all other vertices have trivial vertex groups;
 \item every leaf of $G$ is one of the vertices $\{ v_1, \dots, v_n \}$;
 \item $\iota \colon \pi_1(G) \to A$ is an isomorphism.
\end{enumerate}
The equivalence relation is given by postcomposing $\iota$ with an inner automorphism of $A$, and by multiplying the lengths of all edges of $G$ by a positive constant. We also consider two pairs $G, \iota$ and $G',\iota'$ equivalent if there exists an isometry $\psi \colon G \to G'$ such that $\iota = \iota' \circ \psi$.

Because of the essential uniqueness of the Grushko decomposition, the group $\Out(A)$ acts on $P \mathcal O$ by postcomposing the marking $\iota$.
We offer the following result for this action.

\begin{cor}
\label{fixed point GL}
Let $A$ be a finitely generated group,
and let $H \leqslant \Out(A)$ be a finite subgroup. Then $H$ fixes a barycentre in $P \mathcal O$.
\end{cor}
\begin{proof}
\cref{KPS} gives us an action of the extension $\overline A$ on a tree $T$, and we may assume that this action is minimal; in particular $A$ acts on this tree, and this action satisfies the definition of a vertex in $P \mathcal O$ (with all edge lengths equal to 1).
Since the whole of $\overline A$ acts on $T$, every outer automorphism in $H$ fixes this vertex.
\end{proof}

Note that $P \mathcal O$ has been shown in~\cite[Theorem 4.2, Corollary 4.4]{GuirardelLevitt2007a} to be contractible.

\section{Relative Nielsen realisation}

In this section we use \cref{KPS} to prove relative Nielsen Realisation for free products. To do this we need to formalise the notion of a marking of a space.

\begin{dfn}
We say that a path-connected topological space $X$ with a universal covering $\widetilde X$ is \emph{marked} by a group $A$ \iff it
comes equipped with an isomorphism between $A$ and the group of deck transformations of $\widetilde X$.
\end{dfn}
\begin{rmk}
Given a space $X$ marked by a group $A$, we obtain an isomorphism $A
\cong \pi_1(X,p)$ by choosing a basepoint $\widetilde p \in \widetilde
X$ (where $p$ denotes its projection in $X$).

Conversely, an isomorphism $A \cong \pi_1(X,p)$ together with a choice
of a lift $\widetilde p \in \widetilde X$ of $p$ determines the
marking in the sense of the previous definition.
\end{rmk}

\begin{dfn}
Suppose that we are given an embedding $\pi_1(X) \into \pi_1(Y)$ of fundamental groups of two path-connected spaces $X$ and $Y$, both marked. A map $\iota \colon X \to Y$ is said to \emph{respect the markings via the map $\widetilde \iota$} \iff $\widetilde \iota \colon \widetilde X \to \widetilde Y$ is $\pi_1(X)$-equivariant (\wrt the given embedding $\pi_1(X) \into \pi_1(Y)$), and satisfies the commutative diagram
\[ \xymatrix{ \widetilde X \ar[r]^{\widetilde \iota} \ar[d] & \widetilde Y \ar[d] \\
X \ar[r]^\iota & Y } \]

We say that $\iota$ \emph{respects the markings} \iff such an $\widetilde \iota$ exists.
\end{dfn}

Suppose that we have a metric space $X$ marked by a group $A$, and a group $H$ acting on $X$. Of course such a setup yields the induced action $H \to \Out(A)$, but in fact it does more: it gives us an extension
\[ 1 \to A \to \overline A \to H \to 1 \]
where $\overline A$ is the group of all lifts of elements of $H$ to automorphisms of the universal covering $\widetilde X$ of $X$.

\begin{dfn}
Suppose that we are given a group extension
\[ A \to \overline A \to H \]
We say that an action $\phi \colon H \to \Isom(X)$ of $H$ on a metric space $X$ \emph{realises the extension} $\overline A$ \iff $X$ is marked by $A$, and the extension
\[  \pi_1(X) \to G \to H \]
induced by $\phi$ fits into the commutative diagram
\[ \xymatrix{
 A \ar[d]^{\simeq} \ar[r] & \overline A \ar[r] \ar[d]^{\simeq} & H \ar@{=}[d] \\
 \pi_1(X) \ar[r] & G \ar[r] & H
}
\]

When $A$ is centre-free, and we are given an embedding $H \leqslant \Out(A)$, we say that an action $\phi$ as before \emph{realises the action} $H \to \Out(A)$ \iff it realises the corresponding extension.
\end{dfn}

Now we are ready to state the relative Nielsen Realisation theorem for free products.

\begin{thm}[Relative Nielsen Realisation]
\label{rel NR}
Let $\phi \colon H \to \Out(A)$ be a homomorphism with a finite domain, and let
\[A = A_1 \ast \dots \ast A_n \ast B\]
be a decomposition preserved by $H$, with each $A_i$ finitely generated, and $B$ a (possibly trivial) finitely generated free group.
Let $A_1, \dots, A_m$ be the minimal factors.

Suppose that for each $i \in \{1, \dots, m\}$ we are given a complete NPC space $X_i$ marked by $A_i$, on which $\Stab_i(H)$ acts in such a way that the associated extension of $A_i$ by $\Stab_H(i)$ is isomorphic (as an extension) to the extension $\overline A_i$ coming from $\overline A$.
Then there exists a complete NPC space $X$ realising the action $\phi$, and such that for each $i \in \{1, \dots, m\}$ we have a $\Stab_H(i)$-equivariant embedding $\iota_i \colon X_i \to X$ which preserves the marking.

Moreover, the images of the spaces $X_i$ are disjoint, and collapsing each $X_i$ and its images under the action of $H$ individually to a point yields a graph with fundamental group abstractly isomorphic to the free group $B$.
\end{thm}
As outlined in the introduction, the proof is very similar to the classical
proof of Nielsen realisation, with our new relative Stallings' and
Karrass--Pietrowski--Solitar theorems in place of the classical ones.
\begin{proof}


When $\phi$ is injective we first apply \cref{KPS} to obtain a graph of groups $G$, and then use \cref{blow up} and blow up each vertex of $\widetilde G$ by the appropriate $\widetilde{X_i}$; we call the resulting space $\widetilde X$. The space $X$ is obtained by taking the quotient of the action of $A$ on $\widetilde X$.

If $\phi$ is not injective, then we consider the induced map \[H / \ker \phi \to \Out(A)\] apply the previous paragraph, and declare $H$ to act on the resulting space with $\ker \phi$ in the kernel.
 \end{proof}

\begin{rmk}
In the above theorem the hypothesis on the spaces $X_i$ being complete and  NPC can be replaced by the condition that they are semi-locally simply connected, and any finite group acting on their universal covering fixes at least one point.
\end{rmk}

\begin{rmk}
On the other hand, when we strengthen the hypothesis and require the spaces $X_i$ to be NPC cube complexes (with the actions of our finite groups preserving the combinatorial structure), then we may arrange for $X$ to also be a cube complex.
When constructing the blow ups, we may always take the fixed points of the finite groups to be midpoints of cubes, and then $X$ is naturally a cube complex, when we take the cubical barycentric subdivisions of the complexes $X_i$ instead of the original cube complexes $X_i$.
\end{rmk}

\begin{rmk}
  In \cite{hop} Osajda, Przytycki and the first-named author develop a more topological approach to Nielsen realisation and the
  Karrass--Pietrowski--Solitar theorem. In that article, Nielsen
  realisation is shown first, using \emph{dismantlability} of the sphere
  graph (or free splitting graph) of a free group, and the Karrass--Pietrowski--Solitar theorem then
  follows as a consequence.

  The relative Nielsen realisation theorem with all free factors $A_i$ being finitely generated free groups is a fairly
  quick consequence of the methods developed in \cite{hop} -- however, the
  more general version proved here cannot at the current time be shown
  using the methods of \cite{hop}: to the authors knowledge no analogue of
  the sphere graph exhibits suitable  properties. It would be an interesting
  problem
  to find a ``splitting graph'' for free products which has dismantling
  properties analogous to the ones shown in \cite{hop} to hold for arc, sphere
  and disk graphs.
\end{rmk}

\section{Nielsen realisation for limit groups}

\begin{dfn}
A group $A$ is called \emph{fully residually free} \iff for any finite subset $\{a_1, \dots, a_n \} \subseteq A \s- \{1\}$ there exists a free quotient $q \colon A \to F$ such that $q(a_i) \neq 1$ for each $i$.

A finitely generated fully residually free group is called a \emph{limit group}.
\end{dfn}
Note that the definition immediately implies that limit groups are torsion free.

The three most immediate classes of examples are free groups, free abelian groups, and surface groups. For each of these classes we have a Nielsen realisation results, and these will form a basis for an inductive argument.

\begin{thm}[{\cite{culler1984, Khramtsov1985, Zimmermann1981}}]
\label{NR for free groups}
 Let $H$ be a finite subgroup of $\Out(F_n)$, where $F_n$ denotes the free group of rank $n$. There exists a finite graph $X$ realising the given action $H < \Out(F_n)$.
\end{thm}

\begin{thm}
\label{NR for abelian groups}
 Let
 \[
  \Z^n \to \overline {\Z^n} \to H
 \]
be a finite extension of $\Z^n$. There exists a metric $n$-torus $X$ realising this extension.
\end{thm}

\begin{thm}[Kerckhoff {\cite{Kerckhoff1980, Kerckhoff1983}}]
\label{NR}
 Let $H$ be a finite subgroup of $\Out(\pi_1(\Sigma))$ where $\Sigma$ is a closed surface of genus at least 2. There exists a hyperbolic metric on $\Sigma$ such that $\Sigma$ endowed with this metric realises the given action $H < \Out(\pi_1(\Sigma))$.
\end{thm}

We are going to use an inductive approach to limit groups; such an approach is possible since the class of limit groups coincides with the class of constructible limit groups (for this and other facts see \cite{BestvinaFeighn2009}). What is important for us is that this means that every limit groups has a well-defined \emph{level}, which is a natural number. We do not need to recall the precise definition of the level; it will suffice to recall two facts:
\begin{itemize}
 \item the level of a limit group is equal to $0$ \iff the group is finitely generated and free; and
 \item a limit group $A$ of level $n$ either is a free product of two limit groups of level $\leq n-1$ or it admits a generalised abelian decomposition (which is a special kind of a graph of groups decomposition; see e.g. \cite{BestvinaFeighn2009} for details) and a proper epimorphism $\rho \colon A \to B$ to a limit group of level $\leq n-1$, such that in particular every non-abelian, non-free and non-surface vertex group of the generalised abelian decomposition is mapped injectively by $\rho$.
\end{itemize}

\smallskip
The crucial property of one-ended limit groups is that they admit JSJ-decompositions invariant under automorphisms.

\begin{thm}[Bumagin--Kharlampovich--Myasnikov~{\cite[Theorem 3.13 and Lemma 3.16]{Bumaginetal2007}}]
\label{JSJ}
Let $A$ be a one-ended limit group.  Then there exists a finite graph
of groups $G$ with $\pi_1(G)=A$ and the following additional properties.
\begin{enumerate}[i)]
\item All vertex groups are finitely generated, and all edge groups are cyclic.
\item Every maximal abelian subgroup of $A$ is conjugate to a vertex group of $G$.
\item Each vertex group is either a maximal abelian subgroup of $A$, a \emph{quadratically hanging} subgroup (which implies that it is isomorphic to a finitely generated free group or a surface group) or a \emph{rigid} subgroup. The latter two types are non-abelian.
\item Every edge in $G$ connects a vertex carrying a maximal abelian subgroup to a vertex carrying a non-abelian group (qudratically hanging or rigid).
\item Any automorphism $\phi$ of
$A$ induces an $A$-equivariant isometry $\psi$ of $\widetilde G$ such
that the following diagram commutes
\[ \xymatrix{
A \ar[d]^\phi \ar[r] & \Isom(\widetilde G) \ar[d]^{c_\psi}  \\
A \ar[r] & \Isom(\widetilde G)
}
\]
where $c_\psi$ denotes conjugation by $\psi$.
\end{enumerate}
\end{thm}
We will refer to the graph of groups $G$ as the \emph{canonical JSJ decomposition}. Note that for the naming of the various vertex groups we use the conventions of Bestvina--Feighn \cite{BestvinaFeighn2009}. Note also that the first description of a JSJ decomposition for limit groups very much like the above (without property (v)) can be found in the work of Sela~\cite{Sela2001}.

The canonical JSJ decomposition has one additional property (see \cite[Theorem 3.13(1)]{Bumaginetal2007}) -- it is \emph{universal}, which in particular implies that given any graph of groups decomposition $G'$ of $A$
, every rigid vertex group of the canonical JSJ decomposition is (up to conjugation) contained in a vertex group of $G'$.

\begin{lem}
\label{rigid groups}
Let $A$ be a one-ended limit group of level $n$, and let $R$ be a rigid vertex group in the canonical JSJ decomposition $G$ of $A$. Then $R$ is a surface group or a hyperbolic group isomorphic to a free product of finitely many limit groups of level less than $n$.
\end{lem}
\begin{proof}
Alibegovi\'c \cite[Thm 3.3]{Alibegovic2005} showed that a limit group is hyperbolic \iff its maximal abelian subgroups are cyclic. Let $Z$ be a maximal abelian subgroup of $R$. There exists $Z'$, a maximal abelian subgroup of $A$, which contains $Z$. But $Z'$, when acting on $\widetilde G$, fixes a vertex whose stabiliser is maximal abelian. Thus $Z$ also fixes this vertex, and also the (distinct) vertex corresponding to the nonabelian group $R$. Therefore $Z$ fixes an edge in $\widetilde G$, which implies that $Z$ is cyclic.

The definition of level $n$ tells us that either $A$ is a non-trivial free product (which it cannot be, as it is one-ended), or it admits a graph of groups decomposition $G'$ and an epimorphism $\rho \colon A \to B$ to a limit group of level at most $n-1$ such that $\rho$ is injective on every non-abelian, non-free, non-surface vertex group. If $R$ is a surface group or a free group then we have proven the statement (note that $A$ is not free, and hence $n>0$). If $R$ is not free and not a surface group, then the vertex group in $G'$ in which it lies (up to conjugation) is not free, not abelian, and not a surface group. Thus $\rho$ maps $R$ into $B$ injectively, and so $R$ is a finitely generated subgroup of a limit group of level at most $n-1$. A result of Wilton \cite[Lemma 4.7]{Wilton2009} now tells us that $R$ is isomorphic to a free product of finitely many limit groups of level less than $n$.
\end{proof}

We now need to discuss malnormality.

\begin{dfn}
Recall that a subgroup $G \leqslant A$ is \emph{malnormal} \iff $a^{-1}Ga \cap G \neq \{1\}$ implies that $a \in G$ for every $a \in A$.

Following Sam Brown, we say that a family of subgroups $G_1, \dots, G_n $ of $A$ is \emph{malnormal} \iff for every $a \in A$ we have that $a^{-1}G_ia \cap G_j \neq \{1\}$ implies that $i=j$ and $a \in G_i$.
\end{dfn}

We will use another property of limit groups and their canonical JSJ decompositions.
\begin{prop}[{\cite[Theorem 3.1(3),(4)]{Bumaginetal2007}}]
Let $A$ be a limit group. Every non-trivial abelian subgroup of $A$ lies in a unique maximal abelian subgroup, and every maximal abelian subgroup is malnormal.
\end{prop}

\begin{cor}
 Let $G_v$ be a non-abelian vertex group in a canonical JSJ decomposition of a one-ended limit group $A$. Then the edge groups carried by edges incident at $v$ form a malnormal family in $G_v$.
\end{cor}
\begin{proof}
Let $Z_1$ and $Z_2$ denote two edge groups carried by distinct edges, $e$ and $e'$ say, incident at $v$. Without loss of generality we may assume that each of these groups is infinite cyclic. Suppose that there exists $g \in G_v$ and a non-trivial $z \in g^{-1} Z_1 g \cap Z_2$. Each $Z_i$ lies in a unique maximal subgroup $M_i$ of $A$. But then the abelian subgroup generated by $z$ lies in both $M_1$ and $M_2$, which forces $M_1=M_2$ by uniqueness. Now
\[
 g^{-1} M_1 g \cap M_1 \neq \{1\}
\]
which implies that $g \in M_1$ (since $M_1$ is malnormal), and so $g^{-1} Z_1 g = Z_1$, which in turn implies that $z \in Z_1 \cap Z_2$.

The edges $e$ and $e'$ form a loop in $G$, and so there is the corresponding element $t$ in $A = \pi_1(G)$. Observe that $t$ commutes with $z$, and so the group $\langle t, z \rangle$ must lie in $M_1$. But this is a contradiction, as $t$ does not fix any vertices in $\widetilde G$.
\end{proof}

We are now going to use \cite[Lemma 2.31]{Brown2016}; we are however going to break the argument of this lemma in two parts.

\begin{lem}[Brown]
\label{brown}
Let $X$ be a connected $M_{-1}$-simplicial complex of dimension at most 2. Let $A = \pi_1(X)$, and suppose that we are given a malnormal family $\{ G_1, \dots, G_n \}$ of infinite cyclic subgroups of $A$. Then, after possibly subdividing $X$, each group $G_i$ fixes a (simplicial) axis $a_i$ in the universal cover of $X$, and the images in $X$ of axes $a_i$ and $a_j$ for $i \neq j$ are distinct.
\end{lem}

In the second part of \cite[Lemma 2.31]{Brown2016} we need to introduce an extra component, namely a simplicial action of a finite group $H$ on $X$, which permutes the groups $G_i$ up to conjugation.
\begin{lem}[Brown]
\label{brown 2}
Let $X$ be a locally $\mathrm{CAT}(-1)$ connected finite $M_{-1}$-simplicial complex of dimension at most 2. Let $A = \pi_1(X)$, and suppose that we are given a family $\{ c_1, \dots, c_n \}$ of locally geodesic simplicial closed curves with images pairwise distinct. Suppose that we have a finite group $H$ acting simplicially on $X$ in a way preserving the images of the curves $c_1, \dots, c_n$ setwise. Then there exists a locally $\mathrm{CAT}(k)$ 2-dimensional finite simplicial complex $X'$ of curvature $k$, with $k<0$,  with a transverse family of  locally geodesic simplicial closed curves $\{ c_1', \dots, c_n' \}$, such that $X'$ is $H$-equivariantly homotopic to $X$, and the homotopy takes $c'_i$ to $c_i$ for each $i$.
\end{lem}
\begin{proof}[Sketch of proof]
The proof of \cite[Lemma 2.31]{Brown2016} goes through verbatim, with a slight modification; to explain the modification let us first briefly recount Brown's proof.

We start by finding two local geodesics, say $c_1$ and $c_2$, which contain segments whose union is a tripod -- one arm of the tripod is shared by both segments. We glue in a \emph{fin}, that is a $2$-dimensional $M_{k}$-simplex, so that one side of the simplex is glued to the shared segment of the tripod, and another side is glued to another arm (the intersection of the two sides goes to the central vertex of the tripod). This way one of the curves, say $c_1$, is no longer locally geodesic, and we replace it by a locally geodesic curve identical to $c_1$ except that instead of travelling along two sides of the fin, it goes along the third side.

The problem is that after the gluing of a fin our space will usually not be locally $\mathrm{CAT}(-1)$ (the third side of the fin introduces a shortcut in the link of the central vertex of the tripod). To deal with this, we first replace simplices in $X$ by the corresponding $M_k$-simplices, and this creates an excess angle $\delta$ (compare also the proof of \cref{blow up 2}). Then gluing in the fin does not affect the property of being locally $\mathrm{CAT}(k)$.

We glue such fins multiple times, until all local geodesics intersect transversely; after each gluing we perform a replacement of simplices to generate the excess angle.

\smallskip
Now let us describe what changes in our argument.
When gluing in a fin, we need to do it $H$-equivariantly in the following sense: a fin is glued along two consecutive edges, say $(e,e')$, and $H$ acts on pairs of consecutive edges. We thus glue in one fin for each coset of the stabiliser of $(e,e')$ in $H$. This way, when we introduce shortcuts in a link of a vertex, no two points are joined by more than one shortcut. Since we are gluing multiple fins simultaneously, we need to make the angle $\pi-\delta$ sufficiently close to $\pi$.
\end{proof}
Note that when we say that the family $\{ c_1', \dots, c_n' \}$ is transverse, we mean that each curve $c_i'$ intersects transversely with the other curves and itself.

\begin{thm}
\label{main: NR for limit groups}
Let $A$ be a limit group, and let
\[
 A \to \overline A \to H
\]
be an extension of $A$ by a finite group $H$. Then there exists a complete compact locally CAT($\kappa$) space $X$ realising the extension $\overline A$, where $\kappa = -1$ when $A$ is hyperbolic, and $\kappa=0$ otherwise. When $\kappa=-1$, the space $X$ is of dimension at most $2$.
\end{thm}
\begin{proof}
The proof will be an induction on the level of $A$. Before we start, we will show that assuming the result holds for one-ended limit groups of level $n$, it holds for all limit groups of level $n$.

Consider a limit group $A$ which is not one-ended. We apply the classical version of Stallings theorem to $\overline A$, and split it over a finite group. We will in fact apply the theorem multiple times, so that we obtain a finite graph of groups $G'$ with finite edge groups, with all vertex groups finitely generated and one-ended, and $\pi_1(G') = A$; the fact that we only have to apply the theorem finitely many times follows from finite presentability of $A$ (see~\cite[Theorem 3.1(5)]{Bumaginetal2007}) and Dunwoody's accessibility~\cite{Dunwoody1985}.

The one-ended vertex groups are themselves finite extensions of limit groups, and so for each of them we have a connected metric space to act on by assumption. We now use \cref{blow up} -- the assumption on finite groups fixing points is satisfied since the vertex spaces are complete and CAT(0).\

\smallskip
 We will now assume that $A$, and so $\overline A$, is one-ended.
 As mentioned before, the proof is an induction on the level $n$ of $A$. If $n=0$ then $A$ is a fintely generated free group and we are done by \cref{NR for free groups}. Also, if $A$ is a surface group or a finitely generated free-abelian group, then we are also done by \cref{NR,NR for abelian groups}.

 We apply \cref{JSJ} and obtain a connected graph of groups $G$ with \[\pi_1(G) = A\] (the canonical JSJ decomposition) for which we can extend the natural action of $A$ on $\widetilde G$ to an action of $\overline A$. Taking the quotient by $\overline A$ we obtain a new graph of groups $\Gamma$ with
 \[\pi_1(\Gamma) = \overline A\]
 The edge groups of $\Gamma$ are virtually cyclic, and the vertex groups are finite extensions of finitely generated free or free-abelian groups, or surface groups, or of other rigid vertex groups, which are hyperbolic and isomorphic to free products of finitely many limit groups of smaller level (in view of \cref{rigid groups}).

 Using \cref{NR,NR for free groups,NR for abelian groups} and the inductive hypothesis, for each vertex group $\Gamma_v$ we construct a compact complete NPC space $X_v$ marked by $A_v = A \cap \Gamma_v$, on which $\Gamma_v/A_v$ acts in such a way that the induced extension is isomorphic to $\Gamma_v$. The space $\widetilde X_v$ is CAT($0$). When $G_v$ is hyperbolic, it is actually CAT($-1$) and at most $2$-dimensional. In particular, this is the case for all vertices if $A$ is hyperbolic (and therefore so are all vertex groups, being limit groups whose maximal Abelian subgroups are cyclic).

When $\widetilde X_u$ is (at most) $2$--dimensional and CAT($-1$), we can easily triangulate it $\Gamma_u$-equivariantly using convex simplices;
we can then replace each $2$-simplex by the corresponding $2$-simplex of type $M_{-1}$, and the resulting space is still CAT($-1$).
Thus, $\widetilde X_u$ and $X_u$ have the structure of (at most) $2$-dimensional $M_{-1}$-simplicial complexes, the latter being finite. Moreover, we can triangulate it in such a way that each axis fixed by an infinite cyclic group carried by an edge incident at $u$ is also simplicial. Observe that $\Gamma_v/A_v$ permutes these axes, and so each of the corresponding edge groups in $\Gamma$ preserves such an axis as well.


 Now we apply \cref{brown} and conclude that distinct axes do not coincide. Thus we may use \cref{brown 2}, and replace $X_u$ by a new $\mathrm{CAT}(-1)$ $M_{-1}$-simplicial complex (after rescaling) of dimension at most 2, which has only finitely many isometry classes of simplices, and in which our axes intersect each other and themselves transversely.

Observing that each infinite edge group preserves an axis in each of the relevant vertex spaces by \cref{axis}, we  apply \cref{blow up 2}, and take the resulting space to be $X$. The result is NPC in any case, and if $A$ was hyperbolic (and therefore all vertex spaces were CAT($-1$)), it is locally CAT($-1$) as claimed.
\end{proof}

\bibliographystyle{math}
\bibliography{raags}

\begin{thebibliography}{BKM}

\bibitem[Ali]{Alibegovic2005}
Emina Alibegovi{\'c}.
\newblock {A combination theorem for relatively hyperbolic groups}.
\newblock {\em Bull. London Math. Soc.} {\bf 37}(2005), 459--466.

\bibitem[AB]{AlibegovicBestvina2006}
Emina Alibegovi{\'c} and Mladen Bestvina.
\newblock {Limit groups are {$\rm CAT(0)$}}.
\newblock {\em J. London Math. Soc. (2)} {\bf 74}(2006), 259--272.

\bibitem[BF]{BestvinaFeighn2009}
Mladen Bestvina and Mark Feighn.
\newblock {Notes on {S}ela's work: limit groups and {M}akanin-{R}azborov
  diagrams}.
\newblock In {\em Geometric and cohomological methods in group theory}, volume
  358 of {\em London Math. Soc. Lecture Note Ser.}, pages 1--29. Cambridge
  Univ. Press, Cambridge, 2009.

\bibitem[BH]{bridsonhaefliger1999}
Martin~R. Bridson and Andr{\'e} Haefliger.
\newblock {\em Metric spaces of non-positive curvature}, volume 319 of {\em
  Grundlehren der Mathematischen Wissenschaften [Fundamental Principles of
  Mathematical Sciences]}.
\newblock Springer-Verlag, Berlin, 1999.

\bibitem[Bro]{Brown2016}
Samuel Brown.
\newblock {A gluing theorem for negatively curved complexes}.
\newblock {\em J. Lond. Math. Soc. (2)} {\bf 93}(2016), 741--762.

\bibitem[BKM]{Bumaginetal2007}
Inna Bumagin, Olga Kharlampovich, and Alexei Miasnikov.
\newblock {The isomorphism problem for finitely generated fully residually free
  groups}.
\newblock {\em J. Pure Appl. Algebra} {\bf 208}(2007), 961--977.

\bibitem[Cul]{culler1984}
Marc Culler.
\newblock {Finite groups of outer automorphisms of a free group}.
\newblock In {\em Contributions to group theory}, volume~33 of {\em Contemp.
  Math.}, pages 197--207. Amer. Math. Soc., Providence, RI, 1984.

\bibitem[CV]{cullervogtmann1986}
Marc Culler and Karen Vogtmann.
\newblock {Moduli of graphs and automorphisms of free groups}.
\newblock {\em Invent. Math.} {\bf 84}(1986), 91--119.

\bibitem[Dun1]{Dunwoody1982}
M.~J. Dunwoody.
\newblock {Cutting up graphs}.
\newblock {\em Combinatorica} {\bf 2}(1982), 15--23.

\bibitem[Dun2]{Dunwoody1985}
M.~J. Dunwoody.
\newblock {The accessibility of finitely presented groups}.
\newblock {\em Invent. Math.} {\bf 81}(1985), 449--457.

\bibitem[Gru]{Gruschko1940}
I.~Gruschko.
\newblock {\"Uber die {B}asen eines freien {P}roduktes von {G}ruppen}.
\newblock {\em Rec. Math. [Mat. Sbornik] N.S.} {\bf 8 (50)}(1940), 169--182.

\bibitem[GL]{GuirardelLevitt2007a}
Vincent Guirardel and Gilbert Levitt.
\newblock {The outer space of a free product}.
\newblock {\em Proc. Lond. Math. Soc. (3)} {\bf 94}(2007), 695--714.

\bibitem[HM]{HandelMosher2014}
Michael Handel and Lee Mosher.
\newblock {Relative free splitting and free factor complexes {I}:
  {H}yperbolicity}.
\newblock {\em {arXiv}:1407.3508}.

\bibitem[HK]{HenselKielak2016}
S.~Hensel and D.~Kielak.
\newblock {Nielsen realisation for untwisted right-angled {A}rtin groups}.
\newblock {\em {arXiv}:1410.1618}.

\bibitem[HOP]{hop}
Sebastian Hensel, Damian Osajda, and Piotr Przytycki.
\newblock {Realisation and dismantlability}.
\newblock {\em Geom. Topol.} {\bf 18}(2014), 2079--2126.

\bibitem[Hor]{Horbez2014}
Camille Horbez.
\newblock {The boundary of the outer space of a free product}.
\newblock {\em {arXiv}:1408.0543}.

\bibitem[KPS]{karrassetal1973}
A.~Karrass, A.~Pietrowski, and D.~Solitar.
\newblock {Finite and infinite cyclic extensions of free groups}.
\newblock {\em J. Austral. Math. Soc.} {\bf 16}(1973), 458--466.
\newblock Collection of articles dedicated to the memory of Hanna Neumann, IV.

\bibitem[Ker1]{Kerckhoff1980}
Steven~P. Kerckhoff.
\newblock {The {N}ielsen realization problem}.
\newblock {\em Bull. Amer. Math. Soc. (N.S.)} {\bf 2}(1980), 452--454.

\bibitem[Ker2]{Kerckhoff1983}
Steven~P. Kerckhoff.
\newblock {The {N}ielsen realization problem}.
\newblock {\em Ann. of Math. (2)} {\bf 117}(1983), 235--265.

\bibitem[Khr]{Khramtsov1985}
D.~G. Khramtsov.
\newblock {Finite groups of automorphisms of free groups}.
\newblock {\em Mat. Zametki} {\bf 38}(1985), 386--392, 476.

\bibitem[Kr{\"o}]{Kroen2010}
Bernhard Kr{\"o}n.
\newblock {Cutting up graphs revisited---a short proof of {S}tallings'
  structure theorem}.
\newblock {\em Groups Complex. Cryptol.} {\bf 2}(2010), 213--221.

\bibitem[Sel]{Sela2001}
Zlil Sela.
\newblock {Diophantine geometry over groups I: Makanin-Razborov diagrams}.
\newblock {\em Publications Math{\'e}matiques de l'Institut des Hautes
  {\'E}tudes Scientifiques} {\bf 93}(Sep 2001), 31--106.

\bibitem[Ser]{serre2003}
Jean-Pierre Serre.
\newblock {\em Trees}.
\newblock Springer Monographs in Mathematics. Springer-Verlag, Berlin, 2003.
\newblock Translated from the French original by John Stillwell, Corrected 2nd
  printing of the 1980 English translation.

\bibitem[Sta1]{Stallings1968}
John~R. Stallings.
\newblock {On Torsion-Free Groups with Infinitely Many Ends}.
\newblock {\em Annals of Mathematics} {\bf Second Series, 88}(September 1968),
  312--334.

\bibitem[Sta2]{Stallings1971}
John~R. Stallings.
\newblock {\em Group theory and three-dimensional manifolds}.
\newblock Yale University Press, New Haven,, 1971.

\bibitem[Wil]{Wilton2009}
Henry Wilton.
\newblock {Solutions to {B}estvina \& {F}eighn's exercises on limit groups}.
\newblock In {\em Geometric and cohomological methods in group theory}, volume
  358 of {\em London Math. Soc. Lecture Note Ser.}, pages 30--62. Cambridge
  Univ. Press, Cambridge, 2009.

\bibitem[Zim]{Zimmermann1981}
Bruno Zimmermann.
\newblock {\"{U}ber {H}om\"oomorphismen {$n$}-dimensionaler {H}enkelk\"orper
  und endliche {E}rweiterungen von {S}chottky-{G}ruppen}.
\newblock {\em Comment. Math. Helv.} {\bf 56}(1981), 474--486.

\end{thebibliography}

\bigskip

\noindent
\textsc{Sebastian Hensel} \hfill \textsc{Dawid Kielak} \newline
Mathematisches Institut  \hfill Fakult\"at f\"ur Mathematik  \newline
Universit\"at Bonn  \hfill Universit\"at Bielefeld \newline
Endenicher Allee 60 \hfill Postfach 100131  \newline
D-53115 Bonn \hfill D-33501 Bielefeld \newline
Germany \hfill Germany \newline
\texttt{hensel@math.uni-bonn.de} \hfill \texttt{dkielak@math.uni-bielefeld.de}

\end{document}